\documentclass{amsart}

\usepackage{amsmath}
\usepackage{amssymb}
\usepackage{amsthm}
\usepackage{mathrsfs}
\usepackage{mathtools}
\usepackage[hmarginratio=1:1]{geometry}
\usepackage{graphicx}
\usepackage{longtable}
\usepackage{booktabs}
\usepackage{float}
\usepackage{aliascnt}
\usepackage[font=footnotesize]{caption}
\usepackage{enumitem}

\theoremstyle{definition}
\newtheorem{theorem}{Theorem}[section]

\newtheorem{proposition}[theorem]{Proposition}
\newtheorem{lemma}[theorem]{Lemma}

\newtheorem{remark}[theorem]{Remark}

\numberwithin{equation}{section}
\numberwithin{table}{section}

\allowdisplaybreaks[4]

\newcommand{\NN}{\mathbb{N}}
\newcommand{\ZZ}{\mathbb{Z}}
\newcommand{\QQ}{\mathbb{Q}}

\newcommand{\HH}{\mathbb{H}}
\newcommand{\CC}{\mathbb{C}}

\newcommand{\lcm}{\text{lcm}}

\renewcommand{\pmod}[1]{~(\text{mod }#1)}

\newcommand{\ord}{\text{ord}}

\title{Universal Sums of Triangular Numbers and Squares}
\author{Zichen Yang}
\address{Department of Mathematics, University of Hong Kong, Pokfulam, Hong Kong}
\email{zichenyang.math@gmail.com}
\date{\today}
\keywords{Sums of polygonal numbers, theta series, modular forms}
\subjclass[2020]{11E25, 11F27, 11F30}

\begin{document}

\begin{abstract}
In this paper, we study universal sums of triangular numbers and squares. Specifically, we prove that a sum of triangular numbers and squares is universal if and only if it represents $1,2,3,4,5,6,7,8,10,13,14,15,18,19,20,23,27,28,34,41,47$, and $48$.
\end{abstract}

\maketitle

\section{Introduction}

For an integer $m\geq3$, we define the \textit{generalized $m$-gonal number} as the function $P_m\colon\ZZ\to\ZZ$ given by
$$
P_m(x)\coloneqq\frac{(m-2)x^2-(m-4)x}{2},
$$
with $x\in\ZZ$. A \textit{sum of generalized polygonal numbers} is a function $F\colon\ZZ^r\to\ZZ$ of the form 
\begin{equation}
\label{eqn::polygonalsum}
F(x_1,\ldots,x_r)=a_1P_{m_1}(x_1)+\cdots+a_rP_{m_r}(x_r)
\end{equation}
with integers $1\leq a_1\leq\cdots\leq a_r$ and $m_1,\ldots,m_r\geq3$. We say that an integer $n\in\ZZ$ is \textit{represented} by a sum $F$ of generalized polygonal numbers, or alternatively a sum $F$ of generalized polygonal numbers \textit{represents} an integer $n\in\ZZ$, if there exist $x_1,\ldots,x_r\in\ZZ$ such that $n=F(x_1,\ldots,x_r)$. If a sum of generalized polygonal numbers represents every positive integer, we say that it is a \textit{universal} sum.

For any integer $\mathfrak{M}\geq1$, we denote by $\mathcal{C}_{\mathfrak{M}}$ the class consisting of sums of generalized polygonal numbers of the form (\ref{eqn::polygonalsum}) with parameters $m_1,\ldots,m_r\geq3$ such that $\lcm(m_1-2,\ldots,m_r-2)\leq\mathfrak{M}$. In \cite[Theorem 1.1]{kane2023finiteness}, Kane and the author showed that for any integer $\mathfrak{M}\geq1$, there exists a minimal constant $\Gamma_{\mathfrak{M}}>0$ such that every sum in $\mathcal{C}_{\mathfrak{M}}$ is universal if and only if it represents every positive integer up to the constant $\Gamma_{\mathfrak{M}}$. Moreover, for any real number $\varepsilon>0$, we have an asymptotic upper bound for the constant $\Gamma_{\mathfrak{M}}$ as follows,
\begin{equation}
\label{eqn::asymptoticgrowth}
\Gamma_{\mathfrak{M}}\ll_{\varepsilon}\mathfrak{M}^{43+\varepsilon}.
\end{equation}
As it was pointed out in \cite[Remark 1.2]{kane2023finiteness}, the implied constant in (\ref{eqn::asymptoticgrowth}) is ineffective, because universal ternary sums in $\mathcal{C}_{\mathfrak{M}}$ are not completely classified for any integer $\mathfrak{M}\geq3$. On the other hand, universal ternary sums in $\mathcal{C}_1$ and $\mathcal{C}_2$ were classified in \cite{bosma2013triangular,sun2007mixed,guo2007mixed,oh2009mixed} using various arithmetic methods. Therefore the constants $\Gamma_1$ and $\Gamma_2$ are effectively computable. In fact, Bosma and Kane showed that $\Gamma_1=8$ in \cite{bosma2013triangular} because the class $\mathcal{C}_1$ consists of sums of triangular numbers. The purpose of this paper is to show that $\Gamma_2=48$. 

Since the class $\mathcal{C}_2$ consists of sums of triangular numbers and squares, we establish the following finiteness theorem for such sums.

\begin{theorem}
\label{thm::finitenesstheorem34}
A sum of triangular numbers and squares is universal if and only if it represents every integer in the following set
\begin{equation}
\label{eqn::criticalinteger}
\{1,2,3,4,5,6,7,8,10,13,14,15,18,19,20,23,27,28,34,41,47,48\}.
\end{equation}
In particular, we have $\Gamma_2\leq48$.
\end{theorem}

Moreover, we can show that every integer appearing in the set (\ref{eqn::criticalinteger}) is critical in the following sense.

\begin{theorem}
\label{thm::criticalinteger}
For any integer $t$ in the set (\ref{eqn::criticalinteger}), there exists a sum of triangular numbers and squares representing every positive integer except $t$. Therefore we have $\Gamma_2=48$.
\end{theorem}

The proof of the main theorems uses a standard technique of Bhargava \cite{bhargava2000}, which is the so-called escalator tree. To be precise, the \textit{escalator tree} $T$ for sums of triangular numbers and squares is a rooted tree constructed as follows. We define the root to be $F=0$ with depth $0$, and then we construct the nodes of depth $r+1$ from the nodes of depth $r$ inductively for any integers $r\geq0$ as follows. If a node of depth $r$ is universal, then it is a leaf of the tree. If a sum of triangular numbers and squares $F$ is not universal, we define the \textit{truant} $t(F)$ of the sum $F$ to be the least positive integer not represented by $F$ and the truant of the root is defined to be $1$ by convention. The children of a non-universal node $F(x_1,\ldots,x_r)=\sum_{j=1}^ra_jP_{m_j}(x_j)$ of depth $r$ are the sums of triangular numbers and squares of the form
$$
F(x_1,\cdots,x_r)+a_{r+1}P_{m_{r+1}}(x_{r+1})
$$
such that $a_{r}\leq a_{r+1}\leq t(F)$, $3\leq m_{r+1}\leq4$, and an additional restriction that $m_{r}\leq m_{r+1}$ if $a_{r}=a_{r+1}$ to avoid repeated nodes. It is implied by the explicit construction in the rest of the paper, that the escalator tree $T$ has finitely many nodes. Then it follows that the set (\ref{eqn::criticalinteger}) is exactly the same as the set consisting of truants of non-universal nodes in the escalator tree. So to prove Theorem \ref{thm::finitenesstheorem34}, it suffices to prove the universality of universal nodes and calculate the truants of non-universal nodes. 

We can study the nodes of depth $r\leq3$ using elementary methods. The truant of the root node is $1$ by convention. So it is clear that there are two nodes of depth $r=1$, $P_3$ and $P_4$, whose truants are $2$. It follows that the nodes of depth $r=2$ are $P_3+P_3$, $P_3+P_4$, $P_3+2P_3$, $P_3+2P_4$, $P_4+P_4$, $P_4+2P_3$, and $P_4+2P_4$. We note that there is no need to study the node $P_4+2P_3$ and its descendants because of the following observation of Euler.

\begin{lemma}
\label{thm::eulerobservation}
The nodes $P_3+P_3$ and $P_4+2P_3$ represent the same set of integers.
\end{lemma}
\begin{proof}
This is an observation of Euler in communication with Goldbach. See \cite[pp. 458-460]{fuss1843correspondance}.
\end{proof}

So we ignore the node $P_4+2P_3$ and its descendants in the following discussion. It is easy to verify that every node of depth $r=2$ is not universal, whose truant is given in Table \ref{tbl::truant2}.

{\footnotesize\setlength{\tabcolsep}{1pt}
\begin{longtable}{c|cccccc}
\caption{The truants of the non-universal nodes of depth $r=2$.}
\label{tbl::truant2}\\
\toprule
$F$ & $P_3+P_3$ & $P_3+P_4$ & $P_3+2P_3$ & $P_3+2P_4$ & $P_4+P_4$ & $P_4+2P_4$ \\
\midrule
$t(F)$ & $5$ & $8$ & $4$ & $4$ & $3$ & $5$ \\
\bottomrule
\end{longtable}}

Hence we turn to study the nodes of depth $r=3$. There are universal nodes of depth $r=3$ and they were classified in a series of papers \cite{liouville1863nouveaux,sun2007mixed,guo2007mixed,oh2009mixed}. We summarize the results in the next proposition.

\begin{proposition}
\label{thm::subtree123457}
For nodes of depth $r=3$ in the escalator tree $T$, we have:
\begin{enumerate}
\item Every child of the node $P_3+P_3$ is universal except $P_3+P_3+3P_3$, $P_3+P_3+3P_4$, and $P_3+P_3+5P_4$.
\item Every child of the node $P_3+P_4$ is universal except $P_3+P_4+5P_3$, $P_3+P_4+5P_4$, $P_3+P_4+6P_4$, $P_3+P_4+7P_3$, and $P_3+P_4+7P_4$.
\item Every child of the node $P_3+2P_3$ is universal.
\item Every child of the node $P_3+2P_4$ is universal except $P_3+2P_4+3P_3$, $P_3+2P_4+3P_4$, and $P_3+2P_4+4P_4$.
\item Every child of the node $P_4+P_4$ is universal except $P_4+P_4+P_4$, $P_4+P_4+2P_4$, $P_4+P_4+3P_3$, and $P_4+P_4+3P_4$.
\item Every child of the node $P_4+2P_4$ is universal except $P_4+2P_4+2P_4$, $P_4+2P_4+3P_3$, $P_4+2P_4+3P_4$, $P_4+2P_4+4P_4$, $P_4+2P_4+5P_3$, and $P_4+2P_4+5P_4$.
\end{enumerate}
The truants of the non-universal nodes are given in Table \ref{tbl::truant3}.
\end{proposition}
\begin{proof}
See the papers \cite{liouville1863nouveaux,sun2007mixed,guo2007mixed,oh2009mixed} for the classification of universality ternary sums and it is straightforward to calculate the truants of these non-universal nodes.
\end{proof}

{\scriptsize\setlength{\tabcolsep}{1pt}
\begin{longtable}{c|ccccccc}
\caption{The truants of the non-universal nodes of depth $r=3$.}
\label{tbl::truant3}\\
\toprule
Node & $\underline{P_3+P_3+3P_3}$ & $\underline{P_3+P_3+3P_4}$ & $P_3+P_3+5P_4$ & $\underline{P_3+P_4+5P_3}$ & $\underline{P_3+P_4+5P_4}$ & $P_3+P_4+6P_4$ & $P_3+P_4+7P_3$ \\
\midrule
Truant & $8$ & $8$ & $19$ & $13$ & $13$ & $47$ & $20$ \\
\midrule
Node & $P_3+P_4+7P_4$ & $\underline{P_3+2P_4+3P_3}$ & $\underline{P_3+2P_4+3P_4}$ & $P_3+2P_4+4P_4$ & $\underline{P_4+P_4+P_4}$ & $\underline{P_4+P_4+2P_4}$ & $\underline{P_4+P_4+3P_3}$ \\
\midrule
Truant & $20$ & $7$ & $7$ & $20$ & $7$ & $14$ & $6$ \\
\midrule
Node & $\underline{P_4+P_4+3P_4}$ & $\underline{P_4+2P_4+2P_4}$ & $P_4+2P_4+3P_3$ & $\underline{P_4+2P_4+3P_4}$ & $\underline{P_4+2P_4+4P_4}$ & $P_4+2P_4+5P_3$ & $\underline{P_4+2P_4+5P_4}$ \\
\midrule
Truant & $6$ & $7$ & $23$ & $10$ & $14$ & $10$ & $10$ \\
\bottomrule
\end{longtable}}

This finishes the discussion of nodes of depth $r=3$ and we can move on to the study of nodes of depth $r=4$, which is the major task in this paper. To study nodes of depth $r=4$, a combination of arithmetic methods and analytic methods using the theory of modular forms is applied. So the rest of the paper is organized as follows. In Section \ref{sec::arithmeticmethods}, we study the representations by the children of underlined nodes in Table \ref{tbl::truant3} together with another non-underlined node $P_3+P_4+6P_4$ using arithmetic methods. In Section \ref{sec::analyticmethods}, we study the other nodes of depth $r=4$ using an analytic method based on the theory of modular forms. In the last section, we study nodes of depth $r\geq5$ and prove the main theorems. 

\section{Nodes of Depth $r=4$: Arithmetic Methods}
\label{sec::arithmeticmethods}

In this section, we apply arithmetic methods to study the representations of these nodes. Actually, we can prove every result in this section with the analytic method used in the next section. However, the analytic method is usually more time-consuming than arithmetic methods. Thus, we apply arithmetic methods whenever they are applicable.

We adopt the formulation in terms of quadratic forms with congruence conditions. More precisely, for a sum $F$ of triangular numbers and squares, we construct a quadratic form $Q$ with congruence conditions with two integers $\mu\geq1,\rho\in\ZZ$ such that for any integer $n\in\ZZ$, we have
\begin{equation}
\label{eqn::equalrepresentations}
r_F(n)=r_Q(\mu n+\rho),
\end{equation}
where we denote by $r_F(n)$ the number of representations of $n$ by the sum $F$ and denote by $r_Q(n)$ the number of representations of $n$ by the quadratic form $Q$ with congruence conditions for any integer $n\in\ZZ$.

The construction is given by completing the squares in the sum $F$. If $F=a_1P_3+\cdots+a_{r_1}P_3+b_1P_4+\cdots+b_{r_2}P_4$ is a sum of triangular numbers and squares with $r_1\geq1$, then we take the quadratic form $Q$ with congruence conditions as follows
$$
Q(x_1,\ldots,x_{r_1},y_1,\ldots,y_{r_2})\coloneqq a_1(2x_1+1)^2+\cdots+a_{r_1}(2x_{r_1}+1)^2+2b_1(2y_1)^2+\cdots+2b_{r_2}(2y_{r_2})^2,
$$
with $\mu\coloneqq8$ and $\rho\coloneqq a_1+\cdots+a_{r_1}$. If $F$ is a sum of squares, it is a quadratic form without congruence conditions. In this case, we simply take $Q=F$ with $\mu=1$ and $\rho=0$. For both cases, it is to verify that (\ref{eqn::equalrepresentations}) holds for any integer $n\in\ZZ$.

The arithmetic method used in the next proposition is based on the fact that the quadratic form $Q$ with congruence conditions constructed above corresponding to each underlined node in Table \ref{tbl::truant3} is of class number one, which means that the genus of $Q$ consists of one class, that is, the class of $Q$. In this case, by local-global principle \cite[Corollary 4.8]{chan2013representations}, an integer is represented by $Q$ if and only if it is represented by $Q$ over $\ZZ_p$ for any prime number $p$. Therefore we can find a large set of positive integers represented by $Q$ by local computations. With this knowledge, we then can determine the sets of positive integers represented by the children of this non-universal node.

\begin{proposition}
\label{thm::easysubtrees}
Every child of the underlined nodes in Table \ref{tbl::truant3} is universal, except the children $P_3+P_4+5P_3+10P_3$, $P_3+P_4+5P_3+10P_4$, $P_3+P_4+5P_4+5P_4$, and $P_4+2P_4+5P_4+5P_4$. For these non-universal children, we have
\begin{enumerate}
\item $P_3+P_4+5P_3+10P_3(\ZZ^4)\supseteq\{n\in\NN~|~n\neq23\text{ and }n\not\equiv93,123\pmod{125}\}$.
\item $P_3+P_4+5P_3+10P_4(\ZZ^4)=\{n\in\NN~|~n\neq23\}$.
\item $P_3+P_4+5P_4+5P_4(\ZZ^4)=\{n\in\NN~|~n\neq18\}$.
\item $P_4+2P_4+5P_4+5P_4(\ZZ^4)=\{n\in\NN~|~n\neq15\}$.
\end{enumerate}
\end{proposition}
\begin{proof}
We prove the universality of the children of the node $P_3+P_3+3P_3$ to illustrate the method. For the children of other nodes, we omit the details.

From Table \ref{tbl::truant3}, we know that the truant of the node $P_3+P_3+3P_3$ is $8$. So we have to show that every child of the form $P_3+P_3+3P_3+a_4P_{m_4}$ with $3\leq a_4\leq8$ and $m_4=3,4$ is universal. By the triangular theorem of eight in \cite{bosma2013triangular}, every child of the form $P_3+P_3+3P_3+a_4P_3$ with $3\leq a_4\leq8$ is universal. It remains to prove the universality for every child of the form $P_3+P_3+3P_3+a_4P_4$ with $3\leq a_4\leq 8$. 

Using the above construction of the quadratic form with congruence conditions, it is equivalent to showing that
\begin{equation}
\label{eqn::equation34}
Q(x,y,z)\coloneqq(2x+1)^2+(2y+1)^2+3(2z+1)^2=8n+5-2a_4(2w)^2,
\end{equation}
is solvable with $x,y,z,w\in\ZZ$ for any positive integer $n$. Because the quadratic form $Q$ with the congruence conditions is of class number one, we see that it represents any positive integer of the form $8n+5$ such that $n\not\equiv8,17,23,26\pmod{27}$ by Hensel's lemma. Therefore the equation (\ref{eqn::equation34}) is solvable with $w=0$ for such integers. Furtheremore, if $n\not\equiv26\pmod{27}$ or $a_4\neq3$, it is solvable with $w=1$ for any positive integer $n\equiv8,17,26\pmod{27}$ such that $n\geq a_4$, and if $a_4=3$, it is solvable with $w=2$ for any positive integer $n\equiv26\pmod{27}$ such that $n\geq4a_4$. A quick calculation reveals that it is solvable for any integer $n\leq4a_4-1$ such that $n\not\equiv23\pmod{27}$. Hence the equation (\ref{eqn::equation34}) is solvable for any integer $n\not\equiv23\pmod{27}$.

It remains to show that it is solvable for any positive integer $n\equiv23\pmod{27}$. Note that in this case the integer $8n+5$ is divisible by $9$. Therefore, if we write $n=27m+23$ for some integer $m\geq0$, then it reduces to solving (\ref{eqn::equation34}) for $n=3m+2$. If $3m+2\not\equiv23\pmod{27}$, then we are done. Otherwise we can repeat this argument until it is solvable. Hence every child of the form $P_3+P_3+3P_3+a_4P_4$ with $3\leq a_4\leq8$ is universal.

In a similar manner, we can prove the universality of universal children and determine sets of positive integers represented by non-universal children of the other underlined nodes in Table \ref{tbl::truant3}. For the underlined nodes in Table \ref{tbl::truant3} which are sums of squares, one may refer to \cite{dickson1927integers} for the characterization of positive integers represented by these nodes.
\end{proof}
\begin{remark}
In fact, using the analytic method given in the next section, we can prove that
$$
P_3+P_4+5P_3+10P_3(\ZZ^4)\supseteq\{n\in\NN~|~8n+16\neq200\cdot25^a\text{ with }a\in\NN\},
$$
though the weaker result in Proposition \ref{thm::easysubtrees} is enough for showing that every child of the node $P_3+P_4+5P_3+10P_3$ in the escalator tree $T$ is universal.
\end{remark}

For proving the next lemma, we need a more advanced arithmetic tool developed in \cite{oh2011ternary}.

\begin{lemma}
\label{thm::oh}
The node $P_3+P_4+6P_4$ represents every positive integer $n\equiv0,1\pmod{5}$.
\end{lemma}
\begin{proof}
It is equivalent to showing that
$$
Q(x,y,z)\coloneqq x^2+3y^2+8z^2=8n+1,
$$
is solvable with integers $x\equiv1\pmod{2}$ and $y\equiv0\pmod{4}$ for any positive integer $n\equiv0,1\pmod{5}$. The congruence conditions are superfluous because $x^2+3y^2+8z^2=8n+1$ implies that $x\equiv1\pmod{2}$ and $y\equiv0\pmod{4}$ by passing to $\ZZ/8\ZZ$. The genus of the quadratic form $Q$ consists of two classes, one of which contains the quadratic form $Q_1(x,y,z)\coloneqq Q(x,y,z)$ and the other contains another quadratic form $Q_2(x,y,z)\coloneqq3x^2+3y^2+4z^2+2xy-2xz+2yz$.

We claim that $Q_1$ represents any positive integer $n\equiv1\pmod{5}$ not of the form $n\neq4t^2$ with any integer $t\in\ZZ$. It is easy to verify that any positive integer $n\equiv1\pmod{5}$ is locally represented by $Q_1$. If $Q_1$ represents $n$, then we are done. Otherwise we see that $Q_2$ represents $n$ by local-global principle \cite[Corollary 4.8]{chan2013representations}. Set
\begin{align*}
R_1&\coloneqq\overline{\pm(0,0,2)}\cup\overline{\pm(1,2,3)}\cup\overline{\pm(1,2,4)}\cup\overline{\pm(2,4,0)}\cup\overline{\pm(2,4,4)},\\
R_2&\coloneqq\overline{\pm(2,0,3)}\cup\overline{\pm(2,1,1)}\cup\overline{\pm(2,2,4)}\cup\overline{\pm(2,3,2)},\\
R_3&\coloneqq\overline{\pm(1,3,0)}\cup\overline{\pm(1,3,4)}\cup\overline{\pm(2,1,2)},\\
R_4&\coloneqq\overline{\pm(0,2,2)}\cup\overline{\pm(2,2,1)},\\
R_5&\coloneqq\overline{\pm(2,3,0)},
\end{align*}
where $\overline{(a,b,c)}$ stands for the coset $(a,b,c)+5\ZZ^3$ in $\ZZ^3/5\ZZ^3$. Since $n\equiv1\pmod{5}$, it is not hard to show that any representation $(x,y,z)$ of $n$ lies in the union $R_1\cup R_2\cup R_3\cup R_4\cup R_5$. By a computer searching, we have the following identities relating $Q_1$ with $Q_2$ as follows,
\begin{align*}
I_1\colon&~Q_1\left(\frac{4x+8y+5z}{5},\frac{-3x-y+5z}{5},\frac{-2x+y}{5}\right)=Q_2(x,y,z),\\
I_2\colon&~Q_1\left(\frac{4x+8y-z}{5},\frac{-3x-y-3z}{5},\frac{-2x-y+3z}{5}\right)=Q_2(x,y,z),\\
I_3\colon&~Q_1\left(\frac{8x+4y-5z}{5},\frac{-x-3y-5z}{5},\frac{-x+2y}{5}\right)=Q_2(x,y,z),\\
I_4\colon&~Q_1\left(\frac{8x+4y+z}{5},\frac{-x-3y+3z}{5},\frac{-x+2y+3z}{5}\right)=Q_2(x,y,z).
\end{align*}
For any $1\leq i\leq 4$, assuming that $(x,y,z)\in R_i$ is a representation of $n\equiv1\pmod{5}$ by $Q_2$, then we can construct a representation of $n$ by $Q_1$ via the identity $I_i$. Therefore if $Q_2$ admits a representation of $n\equiv1\pmod{5}$ in $R_1\cup R_2\cup R_3\cup R_4$, then we are done. 

So it remains to deal with representations $(x,y,z)\in R_5$ by $Q_2$. We set
$$
T\coloneqq\frac{1}{5}
\begin{pmatrix}
0  & -5 &  0\\
-1 &  4 &  6\\
-4 & -4 & -1\\
\end{pmatrix}.
$$
Suppose that $(x,y,z)\in\overline{\pm(2,3,0)}$ is a representation of $n$ by $Q_2$. Then it is easy to verify that $T(x,y,z)$ is another representation of $n$ by $Q_2$ lying in $R_1\cup R_2\cup R_3\cup R_4\cup R_5$ by the following identity
$$
Q_2\left(T(x,y,z)\right)=Q_2\left(-y,\frac{-x+4y+6z}{5},\frac{-4x-4y-z}{5}\right)=Q_2(x,y,z).
$$
If $T^k(x,y,z)$ is a representation in $R_1\cup R_2\cup R_3\cup R_4$ for any positive integer $k\geq1$, then we are done. Otherwise by \cite[Lemma 2.1]{oh2011ternary}, the representation $(x,y,z)$ lies in a one-dimensional eigenspace of $T$ with eigenvalue $\det(T)$. Since the primitive eigenvector in this eigenspace is $(1,-1,0)$ and $Q_2(1,-1,0)=4$, we see that any positive integer $n\equiv1\pmod{5}$ not of the form $4t^2$ with any integer $t\in\ZZ$ is represented by $Q_1$. Arguing in a similar manner with the same collection of identities, one can show that any positive integer $n\equiv-1\pmod{5}$ not of the form $4t^2$ with any integer $t\in\ZZ$ is represented by $Q_1$.

Back to the proof of the lemma, it is easy to see that $8n+1\equiv\pm1\pmod{5}$ if and only if $n\equiv0,1\pmod{5}$ and every integer $8n+1$ is not of the form $4t^2$ with any integer $t\in\ZZ$. Therefore the node $P_3+P_4+6P_4$ represents any positive integer $n\equiv0,1\pmod{5}$.
\end{proof}

\begin{proposition}
\label{thm::partofsubtree2-3}
Every child of the node $P_3+P_4+6P_4$ of the form $P_3+P_4+6P_4+a_4P_3$ such that $7\leq a_4\leq 47$ and $a_4\equiv1,4\pmod{5}$, or of the form $P_3+P_4+6P_4+a_4P_4$ such that $6\leq a_4\leq47$ and $a_4\equiv2,3\pmod{5}$ is universal.
\end{proposition}
\begin{proof}
This is straightforward following Proposition \ref{thm::oh}.
\end{proof}

\section{Nodes of Depth $r=4$: Analytic Methods}
\label{sec::analyticmethods}

For the remaining nodes of depth $r=4$, we apply analytic methods to study the representations. Because any remaining node $F=a_1P_3+\cdots+a_{r_1}P_3+b_1P_4+\cdots+b_{r_2}P_4$ such that $r_1+r_2=r$ contains at least a multiple of an triangular number $P_3$, we can associate to it the quadratic form $Q$ with congruence conditions as follows,
$$
Q(x_1,\ldots,x_{r_1},y_1,\ldots,y_{r_2})\coloneqq a_1(2x_1+1)^2+\cdots+a_{r_1}(2x_{r_1}+1)^2+2b_1(2y_1)^2+\cdots+2b_{r_2}(2y_{r_2})^2,
$$
with integers $\mu\coloneqq8$ and $\rho\coloneqq a_1+\cdots+a_{r_1}$, constructed in Section \ref{sec::arithmeticmethods}. So we have the relation
$$
r_F(n)=r_Q(\mu n+\rho),
$$ 
for any integer $n\in\ZZ$.

Let $A$ be the diagonal matrix with entries $2a_1,\ldots,2a_{r_1},4b_1,\ldots,4b_{r_2}$, which is the Hessian matrix of the quadratic form $Q$ with congruence conditions. The \textit{discriminant} $D$ of $Q$ is defined as the determinant of the matrix $A$. The \textit{level} $N$ of $Q$ is defined as the least positive integer $N$ such that diagonal entries of the matrix $NA^{-1}$ are even and off-diagonal entries of the matrix $NA^{-1}$ are integral.

Let $\HH$ be the complex upper half-plane. The \textit{theta series} $\Theta_Q\colon\HH\to\CC$ associated to the quadratic form $Q(x_1,\ldots,x_{r_1},y_1,\ldots,y_{r_2})$ with congruence conditions is defined as
$$
\Theta_Q(\tau)\coloneqq\sum_{x_i,y_j\in\ZZ}e^{2\pi i\tau Q\left(x_1,\ldots,x_{r_1},y_1,\ldots,y_{r_2}\right)}=\sum_{n\geq0}r_Q(n)e^{2\pi in\tau},
$$
for any complex number $\tau\in\HH$. By \cite[Proposition 2.1]{shimura1973modular}, the theta series $\Theta_Q$ is a modular form of weight $2$ for the congruence subgroup $\Gamma_0(4N)$ with the Nebentypus $\chi_D$, where the character $\chi_D$ is given by the Kronecker-Jacobi symbol defined in \cite[p. 442]{shimura1973modular}. By the theory of modular forms, it splits into a sum of an Eisenstein series $E_Q$ and a cusp form $G_Q$. Let $a_{E_Q}(n)$ and $a_{G_Q}(n)$ be the $n$-th Fourier coefficients of $E_Q$ and $G_Q$, respectively. Clearly we have
$$
r_Q(n)=a_{E_Q}(n)+a_{G_Q}(n),
$$
for any integer $n\in\ZZ$. To prove that $Q$ represents the integer $\mu n+\rho$, a typical approach is to bound the Fourier coefficient $a_{E_Q}(\mu n+\rho)$ from below and bound the Fourier coefficient $a_{G_Q}(n)$ from above for any integer $n\geq0$. We give the bounds in Proposition \ref{thm::eisenstein4} and Proposition \ref{thm::cuspidal4}.

\begin{proposition}
\label{thm::eisenstein4}
Suppose that $F=\sum_{i=1}^{4}a_iP_{m_i}$ is a sum of triangular numbers and squares such that $m_i=3$ for at least one $1\leq i\leq 4$. Let $Q$ be the corresponding quadratic form with congruence conditions of discriminant $D$ and level $N$, together with integers $\mu$ and $\rho$ constructed in Section \ref{sec::arithmeticmethods}. If $Q$ represents any integer $\mu n+\rho$ over $\ZZ_p$ for any prime number $p$, then there exists a constant $C_E>0$ such that
$$
a_{E_Q}(\mu n+\rho)\geq C_E\cdot\sigma_{\chi_D}(\mu n+\rho)\cdot(\mu n+\rho),
$$
for any integer $n\geq0$ such that $\ord_p(\mu n+\rho)\leq1$ for any odd prime number $p\mid N$, where $\sigma_{\chi}(n)$ is a twisted divisor function defined as,
$$
\sigma_{\chi}(n)\coloneqq\sum_{d\mid n}\frac{\chi(d)}{d},
$$
for a Dirichlet character $\chi$ modulo $N$. Moreover, the constant $C_E$ has an explicit expression given as follows,
$$
C_E\coloneqq\frac{\pi^2}{L(2,\chi_D)\sqrt{D}}\prod_{p\mid N}\left(\frac{b_p}{1-\chi_D(p)p^{-2}}\right),
$$
with a positive rational number $b_p>0$ depending only on $Q$, $\mu$, and $\rho$ for any prime number $p\mid N$.
\end{proposition}
\begin{proof}
We use the Siegel--Minkowski formula to rewrite the Fourier coefficient $a_{E_Q}(n)$ in terms of a product of local densities. We use the formulation of the Siegel--Minkowski formula given in \cite[(1.15)]{shimura2004inhomogeneous}. Then we have
\begin{equation}
\label{eqn::siegelminkowski}
a_{E_Q}(n)=\frac{4\pi^2n}{\sqrt{D}}\prod_{p}\beta_p(n;Q),
\end{equation}
for any integer $n\geq0$, where if $Q$ is a quadratic form with congruence conditions of the form
$$
Q(x_1,x_2,x_3,x_4)=\sum_{1\leq i,j\leq4}a_{i,j}(\mu_ix_i+\rho_i)(\mu_jx_j+\rho_j),
$$
with integers $a_{i,j},\mu_i,\rho_i\in\ZZ$ for any $1\leq i,j\leq4$, then the local density $\beta_p(n;Q)$ is given by
$$
\beta_p(n;Q)=\int_{\QQ_p}\int_{\prod_{i=1}^4\left(\rho_i+\mu_i\ZZ_p\right)}e_p\Bigg(\sigma\Bigg(\sum_{1\leq i,j\leq4}a_{i,j}x_ix_j-n\Bigg)\Bigg)\mathrm{d}x_1\cdots\mathrm{d}x_4\mathrm{d}\sigma,
$$
for any prime number $p$ and any $p$-adic integer $n\in\ZZ_p$, where the Haar measure on the field $\QQ_p$ is normalized so that the subset $\ZZ_p$ is of measure $1$ and the function $e_p\coloneqq\QQ_p\to\CC$ is defined by $e_p(\alpha)\coloneqq e^{-2\pi ia}$ for some rational number $a\in\bigcup_{t=1}^{\infty}p^{-t}\ZZ$ such that $\alpha-a\in\ZZ_p$. 

To prove the lower bound, we evaluate the local densities $\beta_p(\mu n+\rho;Q)$ for any prime number $p$. Plugging in the quadratic form $Q$ with congruence conditions corresponding to the sum $F=\sum_{i=1}^{4}a_iP_{m_i}$ and changing variables, we see that
$$
\beta_p(\mu n+\rho;Q)=p^{-\ord_p(4)}I_p(2n;\phi_Q),
$$
where $\phi_Q(x_1,\ldots,x_4)=\sum_{i=1}^{4}2a_iP_{m_i}(x_i)$ and $I_p(n;\phi)$ is defined as
$$ 
I_p(n;\phi)=\int_{\QQ_p}\int_{\ZZ_p^4}e_p\big(\sigma\left(\phi(x_1,\ldots,x_4)-n\right)\big)\mathrm{d}x_1\cdots\mathrm{d}x_4\mathrm{d}\sigma,
$$
for any prime number $p$, any $p$-adic integer $n\in\ZZ_p$, and any polynomial $\phi\in\ZZ_p[x_1,\ldots,x_4]$. Then we use the explicit formulae given in \cite[Theorem 4.2 and Theorem 4.6]{kane2023finiteness} to bound the integral $I_p(2n;\phi_Q)$ for any prime number $p$.

If $p$ is a prime number such that $p\nmid N$, we can replace the polynomial $\phi_Q$ by a unimodular quaternary diagonal quadratic form with discriminant $D$ by changing variables in the integral $I_p$. Therefore, applying \cite[Theorem 4.2]{kane2023finiteness}, we have
$$
\beta_p(\mu n+\rho;Q)=I_p(2n;\phi_Q)=\left(1-\frac{\chi_D(p)}{p^2}\right)\sum_{k=0}^{t}\frac{\chi_D(p)^k}{p^k},
$$
where we set $t\coloneqq\ord_p(\mu n+\rho)$. If $p$ is an odd prime number such that $p\mid N$, with the assumption $\ord_p(\mu n+\rho)\leq1$, the integral $I_p(2n;\phi_Q)$ is determined by the congruence class of the integer $n$ modulo $p$ by \cite[Theorem 4.2]{kane2023finiteness}. Since $Q$ represents any integer $\mu n+\rho$ over $\ZZ_p$, there exists a positive rational number $b_p>0$ such that
$$
\beta_p(\mu n+\rho;Q)\geq b_p,
$$
for any integer $n\geq0$ such that $\ord_p(\mu n+\rho)\leq1$. If $p=2$, we see that $\mathfrak{t}_d<\infty$ in \cite[Theorem 4.6]{kane2023finiteness} because $m_i=3$ for at least one $1\leq i\leq 4$. Therefore the integral $I_2(2n;\phi_Q)$ is determined by the congruence class of the integer $n$ modulo $2^{\mathfrak{t}_d-1}$. Since $Q$ represents any integer $\mu n+\rho$ over $\ZZ_2$, there exists a positive rational number $b_2>0$ such that
$$
\beta_2(\mu n+\rho;Q)\geq b_2,
$$
for any integer $n\geq0$. Plugging these bounds into the Siegel--Minkowski formula (\ref{eqn::siegelminkowski}), we obtain the desired lower bound for the Eisenstein part.
\end{proof}

\begin{proposition}
\label{thm::cuspidal4}
Suppose that $F=\sum_{i=1}^{4}a_iP_{m_i}$ is a sum of triangular numbers and squares. Let $Q$ be the corresponding quadratic form with congruence conditions of level $N$, constructed in Section \ref{sec::arithmeticmethods}. Then there exists a constant $C_G>0$ such that for any integer $n\geq0$, we have
$$
|a_{G_Q}(n)|\leq C_G\cdot\sigma_0(n)\cdot n^{\frac{1}{2}},
$$
where $\sigma_0(n)$ is the number of divisors of an integer $n$. Moreover, the constant $C_G$ is given by
$$
C_G\coloneqq\sum_{i\in I}\sum_{\sigma\colon K_i\to\CC}\frac{|\sigma(\gamma_i)|}{\sqrt{d_i}},
$$
if we have
\begin{equation}
\label{eqn::newformsplitting}
G_Q(\tau)=\sum_{i\in I}\sum_{\sigma\colon K_i\to\CC}\sigma(\gamma_i)g_i|_{V(d_i)}^{\sigma}(\tau),
\end{equation}
where $\gamma_i\in\CC$ is a complex number, $d_i$ is an integer, and $g_i$ is a newform of weight $2$ and level $N_i$ such that $N_id_i\mid4N$ for each $i\in I$, and each inner sum runs over the set of embeddings of the base field $K_i$ of the newform $g_i$ into the field $\CC$ of complex numbers.
\end{proposition}
\begin{proof}
This follows from Deligne's optimal bound for the Fourier coefficients of newforms \cite{deligne1974conjecture}.
\end{proof}

From the expressions of the lower bound and the upper bound, we have to bound the quotient of divisor functions to compare the Eisenstein part with the cuspidal part.

\begin{proposition}
\label{thm::quotientdivisor4}
Let $\chi$ be a Dirichlet character. For any real number $\varepsilon>0$, we have
$$
\frac{\sigma_{\chi}(n)}{\sigma_0(n)}\geq C_{\varepsilon}n^{-\varepsilon},
$$
for some constant $C_{\varepsilon}>0$.
\end{proposition}
\begin{proof}
We use the trick of Ramanujan to bound the quotient of the divisor functions. Note that
\begin{equation}
\label{eqn::ratioinn}
\frac{n^{\varepsilon}\sigma_{\chi}(n)}{\sigma_0(n)}
\end{equation}
is multiplicative for any real number $\varepsilon>0$. So it suffices to bound
\begin{equation}
\label{eqn::ratioinpt}
\frac{p^{\varepsilon t}\sigma_{\chi}(p^t)}{\sigma_0(p^t)},
\end{equation}
for any prime number $p$ and any integer $t\geq1$. Because the numerator in the ratio (\ref{eqn::ratioinn}) dominates the denominator as the integer $n$ is sufficiently large, it suffices to bound for finitely many prime numbers $p$, which depends on the fixed real number $\epsilon>0$. For each of these prime numbers $p$, the ratio (\ref{eqn::ratioinpt}) is an increasing function in the variable $t$ as $t$ is sufficiently large. Therefore, we can bound the ratio (\ref{eqn::ratioinn}) from below by a constant $C_{\varepsilon}>0$.
\end{proof}

For each of the remaining nodes, it is easy to verify that the assumptions of Proposition \ref{thm::eisenstein4} holds and it follows that the constant $C_E$ is always positive. Thus, combining Proposition \ref{thm::eisenstein4}, Proposition \ref{thm::cuspidal4}, and Proposition \ref{thm::quotientdivisor4}, the quadratic form $Q$ with congruence conditions represents any sufficiently large positive integer of the form $\mu n+\rho$ such that $\ord_p(\mu n+\rho)\leq1$ for any odd prime number $p\mid N$. Then it remains to check the representability of finitely many positive integers, which can be done by a computer. With this method, we study the remaining nodes of depth $r=4$ in the next proposition.

\begin{proposition}
\label{thm::hardsubtrees}
Every child of the non-underlined nodes in Table \ref{tbl::truant3} is universal except the nodes $P_3+P_4+7P_3+7P_3$, $P_3+P_4+7P_3+14P_3$, $P_3+P_4+7P_3+14P_4$, $P_3+P_4+7P_4+14P_3$, $P_3+P_4+7P_4+7P_4$, $P_3+P_4+7P_4+14P_4$, $P_4+2P_4+5P_3+10P_3$, $P_4+2P_4+5P_3+8P_4$, and $P_4+2P_4+5P_3+10P_4$. For these non-universal children, we have
\begin{enumerate}
\item $P_3+P_4+7P_3+7P_3(\ZZ^4)\supseteq\{n\in\NN~|~8n+15\neq343\cdot49^a\text{ with }a\in\NN\}$.
\item $P_3+P_4+7P_3+14P_3(\ZZ^4)\supseteq\{n\in\NN~|~8n+22\neq294\cdot49^a,686\cdot49^a\text{ with }a\in\NN\}$.
\item $P_3+P_4+7P_3+14P_4(\ZZ^4)\supseteq\{n\in\NN~|~8n+8\neq280\cdot49^a\text{ with }a\in\NN\}$.
\item $P_3+P_4+7P_4+14P_3(\ZZ^4)\supseteq\{n\in\NN~|~8n+15\neq343\cdot49^a\text{ with }a\in\NN\}$.
\item $P_3+P_4+7P_4+7P_4(\ZZ^4)\supseteq\{n\in\NN~|~8n+1\neq217\cdot49^a,385\cdot49^a\text{ with }a\in\NN\}$.
\item $P_3+P_4+7P_4+14P_4(\ZZ^4)\supseteq\{n\in\NN~|~8n+1\neq329\cdot49^a\text{ with }a\in\NN\}$.
\item $P_4+2P_4+5P_3+10P_3(\ZZ^4)\supseteq\{n\in\NN~|~8n+15\neq175\cdot25^a\text{ with }a\in\NN\}$.
\item $P_4+2P_4+5P_3+8P_4(\ZZ^4)=\{n\in\NN~|~n\neq28\}$.
\item $P_4+2P_4+5P_3+10P_4(\ZZ^4)=\{n\in\NN~|~n\neq20\}$.
\end{enumerate}
\end{proposition}
\begin{proof}
The proofs for these nodes are essentially the same. So we illustrate the method using the node $P_3+P_3+5P_4+19P_3$ and omit the details for other nodes. Let $Q$ be the quadratic form with congruence conditions corresponding to the node $P_3+P_3+5P_4+19P_3$. We have $\mu=8$ and $\rho=21$. By a computer program implemented in SageMath \cite{sagemath}, we find that the constants are given by
$$
C_E\approx0.236,~C_G\approx12.645,~C_{\varepsilon}\approx0.482,
$$
with $\varepsilon=0.25$. Therefore by checking the representability up to $152402970$ with a computer, we see that $Q$ represents every integer of the form $8n+21$ with $n\geq0$ such that $\ord_5(8n+21)\leq1$ and $\ord_{19}(8n+21)\leq1$. It is easy to verify that $Q$ represents the integers $5\cdot5^2,5\cdot19^2,13\cdot5^2$ and $13\cdot19^2$. Hence we can conclude that $Q$ represents every integer of the form $8n+21$ with $n\geq0$, which is equivalent to the universality of the node $P_3+P_3+5P_4+19P_3$.

In principle, the proofs for other nodes are similar. However, naive application of the analytic method to prove the universality of every child of the node $P_3+P_4+6P_4$ requires us to compute newforms of levels up to $9024$ with highly non-trivial characters, which is extremely slow in SageMath. So we apply the following tricks to reduce the level.

Note that the truant of the node $P_3+P_4+6P_4$ is $47$. By Proposition \ref{thm::partofsubtree2-3}, we see that it suffices to prove the universality of the child of the form $P_3+P_4+6P_4+a_4P_{m_4}$ for $7\leq a_4\leq47$ with $a_4\equiv0,2,3\pmod{5}$ when $m_4=3$ and for $6\leq a_4\leq47$ with $a_4\equiv0,1,4\pmod{5}$ when $m_4=4$. 

For any child of the form $P_3+P_4+6P_4+a_4P_3$ with $a_4\equiv1\pmod{2}$, let $Q$ be the quadratic form with congruence conditions corresponding to the node $P_3+P_4+6P_4+a_4P_3$. We have $\mu=8$ and $\rho=1+a_4$ and the level of $Q$ is $48a_4$. So naively we have to compute newforms of levels up to $192a_4$. Since $\rho\equiv0\pmod{2}$, we see that the Fourier coefficient $a_{G_Q}(n)$ are only supported on even integers $n$ and every integer $d_i$ in the splitting (\ref{eqn::newformsplitting}) is divisible by $2$. Hence, it suffices to compute newforms of level up to $96a_4$.

For any child of the form $P_3+P_4+6P_4+a_4P_4$ with $a_4\equiv1\pmod{2}$, let $Q(x,y,z,w)\coloneqq (2x+1)^2+2(2y)^2+12(2z)^2+2a_4(2w)^2$ be the quadratic form with congruence conditions corresponding to the node $P_3+P_4+6P_4+a_4P_4$. We have $\mu=8$ and $\rho=1$. Applying the analytic method naively to study the representations of $8n+1$ by $Q$ requires us to compute newforms of levels up to $192a_4$. Note that the quadratic form $Q(x,y,z,w)$ with congruence conditions represents an integer $8n+1$ if and only if the quadratic form $Q'(x,y,z,w)\coloneqq x^2+2y^2+12z^2+8a_4w^2$ represents the integer $8n+1$. Thus, we can apply the analytic method to the quadratic form $Q'$ without congruence conditions, which only requires us to compute newforms of level up to $96a_4$.

Using these tricks, we can reduce the computational overhead for most of difficult cases. Overall it takes several weeks using paralleled computation in a $4$-processor CPU of $4.20$GHz frequency to verify that every child of the node $P_3+P_4+6P_4$ is universal.
\end{proof}

Till now, we have finished the discussion of the nodes of depth $r=4$. The truants of the non-universal nodes of depth $r=4$ are listed in Table \ref{tbl::truant4}.

{\scriptsize\setlength{\tabcolsep}{1pt}
\begin{longtable}{c|ccccc}
\caption{The truants of the non-universal nodes of depth $r=4$.}
\label{tbl::truant4}\\
\toprule
$F$ & $P_3+P_4+5P_3+10P_3$ & $P_3+P_4+5P_3+10P_4$ & $P_3+P_4+5P_4+5P_4$ & $P_4+2P_4+5P_4+5P_4$ & $P_3+P_4+7P_3+7P_3$\\
\midrule
$t(F)$ & $23$ & $23$ & $18$ & $15$ & $41$ \\
\midrule
$F$ & $P_3+P_4+7P_3+14P_3$ & $P_3+P_4+7P_3+14P_4$ & $P_3+P_4+7P_4+14P_3$ & $P_3+P_4+7P_4+7P_4$ & $P_3+P_4+7P_4+14P_4$ \\
\midrule
$t(F)$ & $34$ & $34$ & $41$ & $27$ & $41$ \\
\midrule
$F$ & $P_4+2P_4+5P_3+10P_3$ & $P_4+2P_4+5P_3+8P_4$ & $P_4+2P_4+5P_3+10P_4$ & $\star$ & $\star$\\
\midrule
$t(F)$ & $20$ & $28$ & $20$ & $\star$ & $\star$ \\
\bottomrule
\end{longtable}}

\section{Nodes of Depth $r\geq5$ and Proofs to the Main Theorems}

Finally, we study nodes of depth $r\geq5$ and then it is straightforward to prove the main theorems.

\begin{proposition}
\label{thm::depth56}
Every node of depth $r\geq5$ is universal, except the nodes $P_3+P_4+7P_4+7P_4+21P_3$ and $P_3+P_4+7P_4+7P_4+21P_4$, both of which represent every positive integer except $48$. Therefore any nodes of depth $r=6$ is universal.
\end{proposition}
\begin{proof}
By Proposition \ref{thm::easysubtrees} and Proposition \ref{thm::hardsubtrees}, it is obvious to see that any child of non-universal nodes other than the nodes $P_3+P_4+7P_3+14P_3$ and $P_3+P_4+7P_4+7P_4$ of depth $r=4$ is universal. Now we prove that any child of the node $P_3+P_4+7P_3+14P_3$ is universal. The truant of the node $P_3+P_4+7P_3+14P_3$ is $34$. For any child of the form $P_3+P_4+7P_3+14P_3+a_5P_{m_5}$ with $14\leq a_5\leq34$ and $3\leq m_5\leq4$, it is equivalent to showing that
$$
Q(x,y,z,w)\coloneqq (2x+1)^2+2(2y)^2+7(2z+1)^2+14(2w+1)^2=8n+22-8a_5P_{m_5}(v),
$$
is solvable with $x,y,z,w,v\in\ZZ$ for any integer $n\geq0$. If $8n+22\neq294\cdot49^a$ or $686\cdot49^a$ with any integer $a\geq0$, then the equation is solvable with $v=0$ by Proposition \ref{thm::hardsubtrees}. If $8n+22=294$ or $686$, it is easy to verify that the equation is solvable. If $8n+22=294\cdot49^a,686\cdot49^a$ with any integer $a\geq1$, then the equation is solvable with $v=1$ because $8n+22-8a_5$ is not divisible by $49$ and it is not equal to $294$ or $686$ either. Therefore, any child of the node $P_3+P_4+7P_3+14P_3$ is universal. In a similar fashion, we can show that any child of the node $P_3+P_4+7P_4+7P_4$ is universal except the nodes $P_3+P_4+7P_4+7P_4+21P_3$ and $P_3+P_4+7P_4+7P_4+21P_4$, both of which represent every positive integer except $48$. Since both of them fail to represent only one positive integer, any nodes of depth $r=6$ are universal.
\end{proof}

\begin{proof}[Proof of Theorem \ref{thm::finitenesstheorem34}]
By the construction of the escalator tree, the theorem follows from Table \ref{tbl::truant2}, Table \ref{tbl::truant3}, Table \ref{tbl::truant4}, and Proposition \ref{thm::depth56}.
\end{proof}

\begin{proof}[Proof of Theorem \ref{thm::criticalinteger}]
We know that any integer $t$ in the set (\ref{eqn::criticalinteger}) is the truant of a sum of triangular numbers and squares in the escalator tree, say $F(x_1,\ldots,x_r)$. Suppose that $G$ is a universal sum, for example, $P_3(y_1)+P_3(y_2)+P_3(y_3)$. Then, we show that the sum
$$
H(x_1,\ldots,x_r,y_1,y_2,y_3,z)\coloneqq F(x_1,\ldots,x_r)+(t+1)G(y_1,y_2,y_3)+(2t+1)P_3(z)
$$
represents every positive integer except $t$. Indeed, we notice that the sum $F$ represents every positive integer up to $t-1$ and the sum $(t+1)G$ represents every positive integer that is divisible by $t+1$. Then it follows that $H$ represents every positive integer $n\not\equiv t\pmod{(t+1)}$. Moreover, setting $z=1$, we see that every positive integer $n\equiv t\pmod{(t+1)}$ such that $n\geq2t+1$ is represented by the sum $H$. Therefore the sum $H$ represents every positive integer except $t$. Hence, for any integer $t$ in the set (\ref{eqn::criticalinteger}), there exists a sum of triangular numbers and squares representing every positive integer except $t$.
\end{proof}

\section*{Acknowledgement}

The author is indebted to his advisor Dr. Ben Kane for many helpful suggestions.

\bibliographystyle{plain}
\bibliography{reference}

\end{document}